\documentclass[12pt]{article}
\usepackage[margin=1in]{geometry}
\usepackage{amsthm, amsmath, amsfonts, amssymb}
\usepackage{graphicx}
\usepackage{authblk}

\newcommand{\N}[0]{\mathbb{N}}
\newcommand{\Z}[0]{\mathbb{Z}}

\newcommand{\F}[0]{\mathbb{F}}

\DeclareMathOperator{\constantterm}{ct}
\newcommand{\ct}[1]{\constantterm\left[#1\right]}
\DeclareMathOperator{\coefficient}{coef}
\newcommand{\coef}[2]{\coefficient_{#1}\left[#2\right]}

\newcommand{\x}{\mathbf{x}}
\DeclareMathOperator{\code}{code}

\newtheorem{claim}{Claim}
\newtheorem{conj}{Conjecture}
\newtheorem{question}{Question}
\newtheorem{thm}[claim]{Theorem}
\newtheorem{lemma}[claim]{Lemma}
\newtheorem{prop}[claim]{Proposition}
\newtheorem{cor}[claim]{Corollary}
\theoremstyle{definition}
\newtheorem{definition}{Definition}

\theoremstyle{remark}

\theoremstyle{definition}

\title{A Linear Representation for Constant Term Sequences mod $p^a$ with Applications to Uniform Recurrence}
\author{Nadav Kohen}
\affil{Indiana University\\Bloomington, IN, USA\\nkohen@iu.edu}
\date{February 14, 2025}

\begin{document}
\maketitle

\begin{abstract}
Many integer sequences including the Catalan numbers, Motzkin
numbers, and the Apr{\'e}y numbers can be expressed in the form ConstantTermOf$\left[P^nQ\right]$ for Laurent polynomials $P$ and $Q$. These are often called ``constant term sequences''. In this paper, we characterize the prime powers, $p^a$, for which sequences of this form modulo $p^a$, and others built out of these sequences, are uniformly recurrent. For all other prime powers, we show that the frequency of $0$ is $1$. This is accomplished by introducing a novel linear representation of constant term sequences modulo $p^a$, which is of independent interest.
\end{abstract}

\section{Introduction}
Numerous famous combinatorial sequences admit descriptions of the form $s_n = \ct{P^nQ}$ where $P$ and $Q$ are integral Laurent polynomials (of possibly many variables) and ct extracts the constant term of a Laurent polynomial. For example, the Catalan numbers can be described by $C_n = \ct{(x^{-1} + 2 + x)^n(1-x)}$, the Motzkin numbers can be described by $M_n = \ct{(x^{-1} + 1 + x)^n(1-x^2)}$, and the Apr{\'e}y numbers (which were used to prove the irrationality of $\zeta(3)$) can be described by $A(n) = \ct{\left(\frac{(1+x)(1+y)(1+z)(1+y+z+yz+xyz)}{xyz}\right)^n}$.\\

This paper is primarily focused on showing when sequences mod prime powers are uniformly recurrent, which can be thought of as a weaker form of periodicity:
\begin{definition} A sequence $s_n$ is called \emph{uniformly recurrent} if for every word in $s_n$ (i.e., contiguous subsequence), $w = s_is_{i+1}\cdots s_{i+\ell-1}$, there is a constant $C_w$ such that every occurrence of $w$ is followed by another occurrence of $w$ at a distance of at most $C_w$. I.e., there is a $j\leq C_w$ such that $w = s_{i+j}s_{i+j+1}\cdots s_{i + j + \ell - 1}$. If for all $w$, $C_w$ is bounded by a constant multiple of the length of $w$, then we say that $s_n$ is \emph{linearly recurrent}.
\end{definition}
For background on uniform recurrence in automatic sequences, see Section 10.9 of \cite{Allouche_Shallit_Book} and Section 10.8.8 of \cite{ShallitBook}.\\

In \cite{rowlandzeilberger}, Rowland and Zeilberger give an algorithm for constructing a Deterministic Finite Automaton with Output (DFAO) that computes such constant term sequences modulo a prime power. In \cite{CongruenceAutomaton}, Rampersad and Shallit use these DFAOs in concert with the Walnut library \cite{walnut} to automatically prove theorems about these sequences using computers. In Theorem 10 of that paper the authors prove that, under certain conditions, the sequence $T_n = \ct{(x^{-1} + 1 + x)^n}\bmod p$ is uniformly recurrent if and only if it is never $0$. They then pose Problem 6 in which they conjecture that the Motzkin numbers mod $p$ are uniformly recurrent if and only if $T_n\bmod p$ is never $0$. In \cite{KohenUR}, this author showed that this conjecture is true and that more generally that if $P$ is a symmetric trinomial in one variable, then $\ct{P^nQ}\bmod p$ is uniformly recurrent for every single-variable $Q$ if and only if $\ct{P^n}\bmod p$ is never $0$ (which can be checked by finite means, see Proposition \ref{check_bound}).\\

In Theorem \ref{main_power} and Corollary \ref{combinations_power} of Section 5 of this paper, we extend the results of \cite{KohenUR} to show that $\ct{P^nQ}\bmod p^a$ and every sequence built up of such sequences is linearly recurrent if and only if $\ct{P^n}\bmod p$ is never $0$, and here, no constraints are put on $P$ and both polynomials may have any number of variables. Furthermore, if there exists an $n$ such that $\ct{P^n}\bmod p = 0$ then the frequency of $0$ is $1$ in every $\ct{P^nQ}\bmod p^a$ and every sequence built up of such sequences. (The frequency of a value, $c$, in a sequence, $(s_n)_{n\in\N}$, is equal to the asymptotic density of $\{n\in\N\mid s_n = c\}$).\\

This is accomplished by generalizing the proof method of Theorem 10 of \cite{CongruenceAutomaton} by using constructions from Shallit's book \emph{The Logical Approach to Automatic Sequences} \cite{ShallitBook} along with some novel constructions introduced in Sections 3 and 4. In Section 3, we construct a $p$-linear representation (Definition \ref{linrepdef}) for sequences of the form $\ct{P^nQ}\bmod p$. This $p$-linear representation is in some sense equivalent to the Rowland-Zeilberger Automaton, but it makes many properties of these sequences more apparent and allows us to do away with an assumption made in Theorem 10 of \cite{CongruenceAutomaton}. More generally, this $p$-linear representation seems to be of independent interest as it generalizes a family of Lucas congruences specified in Proposition 1 of \cite{KohenUR}, is amenable to efficient computation associated with constant term sequences mod $p$, and is largely independent of $Q$. This construction is generalized to constant term sequences modulo prime powers in Section 5.

\subsection{Notation and Conventions}
Throughout this paper, $P$ and $Q$ denote Laurent polynomials in $r$ variables with integer coefficients, $\coef{k_1,\ldots,k_r}{Q}$ denotes the coefficient of $x_1^{k_1}\cdots x_r^{k_r}$, $\ct{Q}$ denotes the constant term of $Q$, and $\deg Q$ denotes the largest absolute value of any exponent on any single variable in $Q$ (for example, $\deg(5x^{-3}y^{-2} + 2xy^{-1}) = 3$). For a fixed $P$, we let $a^n_{k_1,\ldots,k_r}$ denote $\coef{-k_1,\ldots,-k_r}{P(x_1,\ldots,x_r)^n} = \ct{P(x_1,\ldots,x_r)^n\cdot x_1^{k_1}\cdots x_r^{k_r}}$ so that the $(r+1)$-dimensional array $a^n_{k_1,\ldots,k_r}$ forms an $(r+1)$-dimensional pyramid that generalizes Pascal's triangle (which is achieved when $P = x^{-1} + x$). To make notation less verbose, we use the bold $\x^{k_1,\ldots,k_r}$ to refer to $x_1^{k_1}\cdots x_r^{k_r}$ and when $k = (k_1,\ldots, k_r)\in \Z^r$, we simply use $\x^k$. Furthermore, for such $k$ we denote by $p\mid k$ that $p$ divides every entry of $k$. Unless otherwise specified, vectors are column vectors and $\vec{u}^{tr}$ is the transpose of $\vec{u}$ (which is a row if $\vec{u}$ is a column).\\

If $\Sigma$ is a set, $\Sigma^*$ denotes the set of words (i.e., strings) of any length whose characters are from $\Sigma$ (including the empty word), and let $\vert w\vert$ denote the length of a word in $w\in\Sigma^*$. If $n$ is a non-negative integer, and $p$ is a prime, then let $n_p\in\F_p^*$ be the word whose characters are the digits of $n$ in base $p$, with the least significant digit first. That is, if we let $n_p[i]$ denote the $i$th digit in the base-$p$ expansion of $n$ so that $n = \sum_{i\in\Z_{\geq 0}} n_p[i]p^i$, then $n_p = (n_p[0])(n_p[1])\cdots(n_p\left[\vert n_p\vert-1\right])$. Please note that in this author's previous paper \cite{KohenUR}, most-significant-digit first (the reverse of this convention) was used, but here we use least-significant-digit first to more closely follow constructions in \cite{ShallitBook}. Also note that when working with strings, exponents denote repetition. For example, $(p-1)^k\in\F_p^*$ denotes a run of $k$ characters that are all the character $(p-1)$. Lastly, note that every statement made in this paper about $n_p$ should also hold for $n_p0^k$ for every $k$.\\

All \emph{morphisms} in this paper refer to morphisms of strings, which are maps, $f: \Sigma_1^*\rightarrow\Sigma_2^*$, such that $f(c_1c_2\cdots) = f(c_1)f(c_2)\cdots$, which can be defined by their action on the elements of $\Sigma_1$. If for every $c\in\Sigma_1$, $\vert f(c)\vert = \ell$ is constant then we say that $f$ is \emph{$\ell$-uniform}. If a morphism is $1$-uniform we call it a \emph{coding}. We say that a morphism, $f:\Sigma^*\rightarrow\Sigma^*$, is \emph{prolongable} on $c\in\Sigma$ if $f(c) = cw$ where $w\in \Sigma^*$ and $f^n(w)$ is non-empty for all $n\geq 0$; in this case, $f^n(c) = cwf(w)f^2(w)\cdots f^{n-1}(w)$ for every $n\geq 0$ and the infinite word $f^\omega(c) = cwf(w)f^2(w)\cdots$ is a fixed point of $f$ (i.e., $f(f^\omega(c)) = f^\omega(c)$).\\

A sequence, $a_n$, is \emph{$p$-automatic} if it has values in some finite alphabet, and there exists a Deterministic Finite Automaton with Output (DFAO) that, given an index $n$ in its base-$p$ representation, $n_p$, ends on a state whose output is $a_n$.\\

A \emph{$p$-linear representation} of rank $m$ of a sequence $a_n$ over $\F_p$ is a triple $(v,\gamma, w)$ consisting of row vector from $\F_p^m$, a matrix-valued morphism, and column vector from $\F_p^m$, respectively. The string morphism $\gamma: \F_p^*\rightarrow \text{Mat}_{m\times m}(\F_p)$ uses matrix multiplication instead of concatenation in its co-domain, such that $v\cdot\gamma(n_p)\cdot w \equiv a_n\pmod p$ for all $n\in\N$. We frequently replace the domain $\F_p$ above with $\Z/p^a\Z$ except in the input alphabet of $\gamma$, which is always $\F_p$. If $a_n$ has such a $p$-linear representation, we say it is \emph{$p$-regular}. Theorem 9.6.1 of \cite{ShallitBook} says that a sequence with finitely many values is $p$-automatic if and only if it is $p$-regular.

\section{Preliminaries}
The Rowland-Zeilberger automaton is wonderfully described, with examples, in \cite{rowlandzeilberger}. Its existence shows that constant term sequences modulo prime powers are $p$-automatic. The general construction is specified in the section titled \emph{Teaching the Computer How to Create Automatic $p$-schemes}. We repeat some of that exposition here:\\

\begin{definition}
For every positive integer $k$, define $\Lambda_k$ to be the map from (and to) Laurent polynomials over $\Z$ (or the induced map on quotient rings) defined by $$\Lambda_k\left(\sum_{i\in [-m,m]^r}q_i\x^i\right) = \sum_{k\mid i}q_i\x^{\frac{i}{k}}.$$ That is, $\Lambda_k$ deletes terms whose exponents are not all multiples of $k$ and then performs the change of variables $x_j^k\mapsto x_j$ for every $j$.
\end{definition}

\begin{lemma}\label{RZ_cong}
For all Laurent polynomials $P$ and $Q$, all primes $p$ and all integers $n$ and $k$, $$\ct{P^{pn+k}\cdot Q}\equiv \ct{P^n\cdot\Lambda_p\left(P^kQ\right)}\pmod p.$$
\end{lemma}
\begin{proof}
First note the equivalence often called the generalized Freshman's dream, $P(x_1,\ldots, x_r)^p\equiv P(x_1^p,\ldots,x_r^p)\pmod p$, and then note that the change of variables $x_i^p\mapsto x_i$ does not affect constant terms. From these two observations, it follows that
\begin{align*}
\ct{P(x_1,\ldots, x_r)^{pn+k}\cdot Q(x_1,\ldots, x_r)} &\equiv \ct{P(x_1^p,\ldots,x_r^p)^n\cdot P(x_1,\ldots, x_r)^kQ(x_1,\ldots, x_r)}\\
&= \ct{P(x_1,\ldots, x_r)^n\cdot\Lambda_p\left(P(x_1,\ldots, x_r)^kQ(x_1,\ldots, x_r)\right)}.
\end{align*}
\end{proof}

Using this congruence, we can construct a Deterministic Finite Automaton with Output (DFAO) whose states are labeled with Laurent polynomials by beginning with the state $Q$, and then having transitions labeled $k$ for $k\in\F_p$ that lead to (possibly new) states labeled by $Q' = \Lambda_p\left(P^kQ\right)$, and repeating this process until no new states can be reached (see the following lemma to understand why this terminates). The output of a state labeled $Q$ is $\ct{Q}$. This DFAO is known as the Rowland-Zeilberger automaton, and its construction constitutes a proof that $a_n = \ct{P^nQ}\bmod p$ is a $p$-automatic sequence.

\begin{lemma}\label{deg_bound}
For a fixed $P$, if we define $\Lambda_p^k(Q) = \Lambda_p(P^kQ)$, then $$m := \max(\deg (P) - 1, \deg (Q))\geq \deg\Lambda_p^{k_t}(\cdots \Lambda_p^{k_0}(Q)\cdots ).$$ That is, $m$ is a bound on the degree of every polynomial that is the result of iterating $\Lambda_p^k$ on input $Q$ for all values of $k\in\F_p$. In particular, $m$ bounds the degree of every state, $Q$, appearing in the Rowland-Zeilberger construction.
\end{lemma}
\begin{proof}
If $Q_0 = \sum_{i\in[-m_0,m_0]^r} q_i\x^i$ is of degree $m_0$, then for all $k\leq p-1$, $\Lambda_p^k$ is $\F_p$-linear so that $Q_1 = \Lambda_p^k(Q_0) = \sum_{i\in[-m_0,m_0]^r}q_i\Lambda_p(P^k\x^i)$ and $\deg Q_1 = \max_i\left(\deg(\Lambda_p(P^k\x^i))\right)$, which depends only on $k, m_0,$ and $P$. For $i\geq 0$, letting $\vert (j_1,\ldots,j_r)\vert = \max_k j_k$, and recalling that $a_{-j}^k = \coef{j}{P^k}$, $$\Lambda_p(P^k\x^i) = \Lambda_p\left(\sum_{\vert j\vert<k\deg P}a^k_{-j}\x^{i+j}\right) = \sum_{\vert j\vert\leq k\deg P, p\mid i+j}a^k_{-j}\x^{\frac{i+j}{p}}.$$ From this, since $\vert j\vert \leq k\deg P\Rightarrow \vert\frac{i+j}{p}\vert\leq \frac{k\deg P + \vert i\vert}{p}\leq \frac{(p-1)\deg P + \vert i\vert}{p}$, we can see that $$\deg Q_1\leq \frac{p-1}{p}\deg P + \frac{\deg Q_0}{p},$$ independently of $k$. And in general, a simple induction shows that every $Q_n$ reached from $Q_0$ in $n$ steps has $$\deg Q_n\leq \deg P + \frac{\deg Q_0 - \deg P}{p^n}.$$ From this, we can immediately see that if $\deg Q_0 < \deg P$ then $\deg Q_n < \deg P$ as well for all $n$. Furthermore, the degrees of our $Q_i$s are decreasing when $\deg Q_0 > \deg P$ since this implies that $\deg Q_1\leq \deg P + \frac{\deg Q_0 - \deg P}{p} = \frac{(p-1)\deg P + \deg Q_0}{p} < \deg Q_0$; in particular, if $\log_p(\deg Q_0 - \deg P) < n$ then $\deg Q_n\leq \deg P$ for every $Q_n$ reachable from $Q_0$ in $n$ steps. Lastly, if $\deg Q_0 = \deg P$, then $\deg Q_1 = \deg\Lambda_p(P^{k}Q_0) < \deg P$ for all $k < p-1$, while $\deg\Lambda_p(P^{p-1}Q_0) = \deg P$.
\end{proof}

For a different analysis yielding Lemma \ref{deg_bound}, see Theorem 2.4 of \cite{statebound}.\\

We use $m$ as defined in Lemma \ref{deg_bound} for the remainder of the paper when $P$ and $Q$ are understood from context.\\

In general, $p$-automatic sequences have (at least) three useful descriptions. First, as a sequence determined by a DFAO as above. Second, as the result of coding a fixed point of a prolongable uniform morphism, whose equivalence to the first description is known as Cobham's Little Theorem \cite{CobhamLittleThm} (or see Theorem 5.4.1 of \cite{ShallitBook}). Third, every $p$-automatic sequence is always $p$-regular and hence has a $p$-linear representation $(v,\gamma,w)$, and the converse ($p$-regular $\Rightarrow$ $p$-automatic) is true for sequences taking on finitely many values. This equivalence is Theorem 9.6.1 of \cite{ShallitBook}, and the proof is constructive: using the $p$-linear representation of our sequence, a prolongable uniform morphism is built whose infinite fixed-point has a coding yielding the same sequence.\\

We make use of all three types of descriptions to prove our results, but the ``free'' constructions are not sufficient for our purposes. Hence, rather than using the Rowland-Zeilberger automaton for $a_n = \ct{P^nQ}\bmod p$ directly, we instead derive a $p$-linear representation for $a_n$ by using Lemma \ref{RZ_cong} directly. This yields a generalization of the Lucas congruences for central binomial and trinomial coefficients, which were the primary tools used in \cite{KohenUR} to derive similar uniform recurrence results when $P$ is of degree at most $1$. We then use this $p$-linear representation to build our other descriptions of $a_n$.

\section{A Novel Linear Representation}
Our $p$-linear representation is the result of interpreting the Rowland-Zeilberger machine within the vector space of Laurent polynomials over $\F_p$ of degree bounded by $m$. Note that this construction differs from the standard construction creating a $p$-linear representation from a DFAO, which uses the states of the DFAO as a basis rather than, for example, the standard basis of powers of $\x$.

\begin{definition}
Recall the definition of $m$ from Lemma \ref{deg_bound}. We wish to encode Laurent polynomials with coefficients in $\F_p$ of degree $\leq m$ as a product of copies of $\F_p$, and so we must choose a computational basis: Impose an order on $T = [-m,m]^r$ (e.g. lexicographical), where $T$ indexes the set of coefficients in a polynomial with degree $\leq m$. If $Q = \sum_{i\in T}q_i\x^i$ then let $\text{vec}(Q)$ be the row vector $(q_j)_{j\in T}$ and let $\code_Q:(\F_p^T)^*\rightarrow\F_p^*$ be the coding defined by $\code_Q(\vec{u}) = \text{vec}(Q)\cdot\vec{u}$.
\end{definition}

\begin{prop}\label{bookkeeping}
For every $k\in\F_p$ and $i\in T$, $\Lambda_p(P^k\x^i) = \sum_{j\in T}a_{i-pj}^k\x^j$.
\end{prop}
\begin{proof}
Using the definition of $\Lambda_p$, $$\coef{j}{\Lambda_p(P^k\x^i)} = \coef{pj}{P^k\x^i} = \coef{pj - i}{P^k} = a_{i-pj}^k.$$
\end{proof}

\begin{definition}\label{linrepdef}
Notice that the map $Q\mapsto \Lambda_p(P^kQ)$ is $\F_p$-linear and maps $\F_p^T$ (the space of bounded-degree Laurent polynomials) to itself by Lemma \ref{deg_bound}. Thus, this map corresponds to a matrix with entries in $\F_p$, which we denote $\gamma(k)$. This map is determined by its action on the standard basis of monomials (written as row vectors to the left of $\gamma(k)$). Therefore, we can deduce from Proposition \ref{bookkeeping} that $$\gamma(k) := (a^k_{i - pj})_{i,j\in T}.$$ The matrix-valued string morphism, $\gamma$, defined in this way along with the row vector \text{vec}($Q$) and the column vector, $V(0)$, whose only non-zero entry is at the index corresponding to the constant term constitute a $p$-linear representation for the sequence $a_n = \ct{P^nQ}\bmod p$. This is because the computation vec$(Q)\cdot \gamma(n_p)\cdot V(0)$ when multiplied from left to right is exactly the computation executed by the Rowland-Zeilberger machine yielding $a_n$.
\end{definition}

Put another way: in Shallit's \emph{The Logical Approach to Automatic Sequences} \cite{ShallitBook}, section 9.3, he describes how to convert a $p$-linear representation of a $p$-automatic sequence into a set of relations of its finite $p$-kernel, and then Theorem 5.5.1 (which shows constructively that a sequence has finite $p$-kernel if and only if it is $p$-automatic) defines a process for converting this finite set of relations into a DFAO describing the sequence. Successively applying these two algorithms to the $p$-linear representation $(\text{vec}(Q), \gamma, V(0))$ exactly yields the Rowland-Zeilberger DFAO described in this section.\\

The matrices, $\gamma(k)$, of Definition \ref{linrepdef} can also be interpreted as acting on the left of column vectors that correspond to linear functionals. That is, we begin with $V(0)$, which corresponds to the linear functional $Q\mapsto \ct{Q}$, and then read the most significant base-$p$ digit, $k$, of $n$ in order to update to the new linear functional $Q\mapsto \ct{P^k Q}$. This process is then repeated until the final linear functional is $Q\mapsto \ct{P^nQ}$, as desired. This observation motivates the following definition and lemma.

\begin{definition}\label{Vdef}
Let $V(n)$ be the column vector $(a^n_i)_{i\in T}\in\F_p^T$, which stores the coefficients of $P^n$ whose terms have degree at most $m$.
\end{definition}

\begin{lemma}\label{Vlemma} For all non-negative integers $n$ and all $k\in\F_p$,
$$\gamma(k)\cdot V(n) = V(pn+k).$$
\end{lemma}
\begin{proof}
In view of Lemma \ref{RZ_cong} and Proposition \ref{bookkeeping}, for all $i\in T$,
\begin{align*}
V(pn+k)[i] &= a^{pn+k}_i\\
&= \ct{P^{pn+k}\x^i}\\
&\equiv \ct{P^n\cdot\Lambda_p(P^k\x^i)}\pmod p\\
&= \ct{P^n\cdot\sum_{j\in T}a^k_{i-pj}\x^j}\\
&= \sum_{j\in T}a^k_{i-pj}\cdot \ct{P^n\x^j}\\
&= \sum_{j\in T}a^k_{i-pj}a^n_j\\
&= \left((a^k_{i-pj})_{i,j\in T}\cdot V(n)\right)[i]\\
&= \left(\gamma(k)\cdot V(n)\right)[i].
\end{align*}
The result follows.
\end{proof}

Remarkably, this lemma shows that modulo $p$, the low-degree coefficients of $P^n$ can be efficiently computed without reference to higher degree coefficients (except for those of $P^k$ with $k<p$). It also details the process executed by our new reverse Rowland-Zeilberger machine whose states are the $V(i)$, which have output $\code_Q(V(i))$. This machine is essentially an instantiation of the reversal of a generic DFAO given in Theorem 4.3.3 of \cite{Allouche_Shallit_Book}, except that we keep the algebraic information inherited from the fact that we are computing constant term sequences. This description also shows that directly computing the reverse Rowland-Zeilberger machine is not, in principle, more expensive than computing the usual machine. Finally, the primary benefit of this description of $\ct{P^nQ}\bmod p$ is that it is independent of $Q$ until the output is computed from the final state. That is, we can study the sequence $(V(n))_{n\in\N}$ whose elements are in $\F_p^T$ and think of all of our constant term sequences as the images of this one sequence under various codings.\\

To this end, we introduce a description of the sequence $(V(n))_{n\in\N}$ as a fixed point of a prolongable uniform morphism.

\begin{definition}\label{morph}
When a Laurent polynomial $P$, a prime $p$, and an integer $m\geq \deg(P) - 1$ (and $m\geq \deg Q$ for every $Q$ that we may be considering) are fixed, we define the $p$-uniform morphism $\sigma:\left(\F_p^T\right)^*\rightarrow\left(\F_p^T\right)^*$ by $$\sigma(\vec{u}) := (\gamma(0)\cdot\vec{u})(\gamma(1)\cdot\vec{u})\cdots(\gamma(p-1)\cdot\vec{u}).$$
$\sigma$ is prolongable on $V(0)$ since $\gamma(0)\cdot V(0) = V(0)$ by Lemma \ref{Vlemma}, and so $$\alpha:= \sigma^\omega(V(0)) = V(0)V(1)V(2)\cdots\in\left(\F_p^T\right)^*$$ whose $n$th character is $V(n)$, is a fixed point of $\sigma$, i.e., $\sigma(\alpha) = \alpha$.\\

To see that $\alpha[n] = V(n)$, note that $$\sigma(V(j)) = (\gamma(0)\cdot V(j))\cdots(\gamma(p-1)\cdot V(j)) = V(pj)\cdots V(pj+p-1)$$ and thus $\sigma(V(0)V(1)\cdots V(p^n-1))$ is equal to $$(V(0)\cdots V(p-1))(V(p)\cdots V(p^2-1))\cdots (V(p^{n+1}-p)\cdots V(p^{n+1}-1)).$$
\end{definition}

In essence, we can fluidly and interchangeably utilize all four (compatible) representations of $\ct{P^nQ}$; as the Rowland-Zeilberger DFAO (Lemma \ref{RZ_cong}), as the reverse Rowland-Zeilberger DFAO (Lemma \ref{Vlemma}), using the $p$-linear representation of Definition \ref{linrepdef}, and using a coding of the fixed point of Definition \ref{morph}. The dictionary allowing us to go between these representations is most easily stated by relating each description to the $p$-linear representation as follows: The states of the Rowland-Zeilberger DFAO are labeled by polynomials $Q$, corresponding to $\text{vec}(Q)$ in the $p$-linear representation; transitions from $Q$ to $Q'$ labeled by $k$ in the DFAO are such that $\text{vec}(Q)\cdot\gamma(k) = Q'$. Likewise, the states of the reverse Rowland-Zeilberger DFAO are labeled by the $V(i)$ with transitions from $V(i)$ to $V(j)$ labeled by $k$ such that $\gamma(k)\cdot V(i) = V(j)$. Lastly, the relation $\gamma(k)\cdot V(n) = V(pn+k)$ describes the progression of the fixed-point, $\alpha$, which can be useful for reasoning about properties such as uniform recurrence.\\

We end this section with a fundamental lemma about $\gamma(0)$:

\begin{lemma}\label{zero_out}
Let $C$ be the index associated with the polynomial $1$ in $T$'s order. There is an integer $s\in\N$ such that for every $s'\geq s$, $\gamma(0)^{s'} = \gamma(0^{s'}) = e_{C,C}$, that is to say it has a $1$ in the row and column corresponding to the constant term and $0$s everywhere else.
\end{lemma}
\begin{proof}
The matrix $e_{C,C}$ is characterized by the fact that for every vector $\vec{u}$, $\vec{u}^{tr}\cdot e_{C,C} = (\vec{u}[C])e_C^{tr}$ since, in particular, this being true for the basis vectors $e_i^{tr}$ implies that all rows other than the $C$ row are zero and that the $C$ row is $e_C^{tr}$. Every row vector can be written as $\text{vec}(Q)$ for some $Q$, and $\text{vec}(Q)\cdot\gamma(0) = \text{vec}(\Lambda_p(P^0Q)) = \text{vec}(\Lambda_p(Q))$. Furthermore, $\Lambda_p(Q)$ has degree at most $\frac{0\cdot\deg P + \deg Q}{p} = \frac{\deg Q}{p}$ as seen in Lemma \ref{deg_bound}. So letting $s'\geq s = \lfloor\log_p(\deg Q)\rfloor + 1$ yields $\text{vec}(Q)\cdot\gamma(0)^{s'} = \text{vec}((\Lambda_p^0)^{s'}(Q)) = \text{vec}(\ct{Q}) = (\text{vec}(Q)[C])e_C^{tr}$, as desired.
\end{proof}

This lemma can be thought of as partially explaining some of the connection between the sequences $\ct{P^nQ}$ and $\ct{P^n\cdot 1}$ since, beginning with $Q$, we can always reach a state corresponding to $Q'\in\F_p\subset\F_p[x_1,\ldots, x_r,x_1^{-1},\ldots,x_r^{-1}]$ in some fixed number of steps using $\gamma(0^s)$.

\section{Classifying Uniformly Recurrent Constant Term Sequences mod $p$}
We begin by analyzing the case in which $\ct{P^nQ}\bmod p$ is not uniformly recurrent, in which case the frequency of $0$ is $1$. First, we note that checking whether a $0$ appears in $\ct{P^n}\bmod p$ can be done by only inspecting a prefix of this sequence.

\begin{prop}\label{check_bound}
Let $P$ be any Laurent polynomial and $p$ be prime. There is a $B_{P,p}\in\N$, depending on $p$ and $\deg(P)$, so that if there exists some $n\in\N$ such that $p\mid \ct{P^n}$, then there exists an $n_0\in\N$ with $n_0<B_{P,p}$ such that $p\mid \ct{P^{n_0}}$.
\end{prop}
\begin{proof}
The Rowland-Zeilberger automaton for $\ct{P^n}\bmod p$ contains at most $p^{\vert T\vert} = p^{(2m+1)^r} = p^{(2\cdot\deg(P)-1)^r}$ states. If one of the states has output $0$, then there must be a path to it from the starting state of length less than the number of states. Thus, if we let $B_{P,p} = p^{p^{(2\cdot\deg(P)-1)^r}}$, then there must be a $0$ in our sequence for some $n_0<B_{P,p}$.
\end{proof}

However, this bound of $p^{p^{(2\cdot\deg(P)-1)^r}}$ seems far too large. In the author's brief computer experimentation in one variable, no sequence or prime was found violating the following bound:

\begin{conj}
Proposition \ref{check_bound} holds for $B_{P,p} = p^{\deg(P)}$ when $P$ has one variable ($r=1$).
\end{conj}

Now, we show one direction of our main result: constant term sequences are not uniformly recurrent when the corresponding $Q=1$ constant term sequence has an entry divisible by $p$.

\begin{thm}\label{zero_density}
If $P$ is any Laurent polynomial and there exists some $n\in\N$ such that $\ct{P^n}\equiv 0\pmod p$, then the frequency of $0$ is $1$ in every sequence of the form $\ct{P^nQ}\bmod p$ for every Laurent polynomial $Q$. That is, the set of indices, $n$, such that $\ct{P^nQ}\equiv 0\pmod p$ has density $1$. In particular, $\ct{P^nQ}\bmod p$ contains arbitrarily large runs of $0$s and is not uniformly recurrent.
\end{thm}
\begin{proof}
If $p\mid \ct{P^{n_0}}$, then there exists $s\in\N$ by Lemma \ref{zero_out} such that $\gamma(0)^s\gamma(n_0)\gamma(0)^s = 0$. Because $\ct{P^nQ} \equiv \text{vec}(Q)\cdot V(n) = \text{vec}(Q)\cdot\gamma(n)\cdot V(0)$, all $n$ such that $n_p$ contains $0^s(n_0)_p0^s$ must have $\ct{P^nQ} \equiv 0$. If $\beta = \vert (n_0)_p\vert$, then consider each $n\in \left[0, p^{(\beta+2s)\cdot k}\right)$ as corresponding to an element of $\left(\F_p^{\beta+2s}\right)^k$ via its base-$p$ representation chunked into blocks of size $\beta+2s$. If any of the $k$ entries in $\F_p^{\beta+2s}$ for $n$ correspond to $0^s(n_0)_p0^s$, then we know that $V(n)=0$ and $\ct{P^nQ}\equiv 0$, so at least $p^{(\beta+2s)\cdot k} - \left(p^{\beta+2s} - 1\right)^k$ of such $n$ have $\ct{P^nQ} \equiv 0$. Therefore, the frequency of $0$ in $\ct{P^nQ}\bmod p$ is at least $\frac{p^{(\beta+2s)\cdot k} - \left(p^{\beta+2s} - 1\right)^k}{p^{(\beta+2s)\cdot k}} = 1 - \left(\frac{p^{\beta+2s} - 1}{p^{\beta+2s}}\right)^k\rightarrow 1$ as $k\rightarrow\infty$. In fact, this shows that the frequency of $0$ in $\alpha$ (of Definition \ref{morph}) is $1$.
\end{proof}

We wish to prove that if $0$ does not appear in $\ct{P^n}\bmod p$, then $\ct{P^nQ}\bmod p$ is uniformly recurrent. Our main result follows from the following proposition about the morphism $\sigma$ of Definition \ref{morph}:

\begin{prop}\label{v0_appears}
There exists $t_0\in\N$ such that for all $n\in\N$, if $p\nmid \ct{P^n}$ then $\sigma^{t_0}(V(n)) = x_0V(0)y_0$ with $x_0,y_0\in\left(\F_p^T\right)^*$. That is, $V(0)$ is a character in $\sigma^{t_0}(V(n))$ for every such $n$.
\end{prop}
\begin{proof}
Let $s\in\N$ such that $\gamma(0)^s = e_{C,C}$. Next, fix $s'\in\N$ such that $C$-term entries (i.e., indexed by $C$) of $\sigma^{s'}(V(0))$ contain all $C$-term values of $V(i)$ that occur for any $i\in\N$. Recall that the $C$-term value of $V(i)$ is the constant coefficient of $P^i$. An $s'$ exists since there are finitely many (namely, $p$) values that $C$-term entries of $V(i)$ can take on, each initially occurring at some finite value of $i$. Let $V(n)[C]\in\F_p$ denote the entry of $V(n)$ corresponding to the constant term, which is not $0$ by our assumption that $p\nmid \ct{P^n}$. Finally, we have that $\gamma(0)^s\cdot V(n) = \left(V(n)[C]\right)V(0)$, $\sigma^{s' + s}(V(0))$ contains $\left(V(n)[C]\right)V(0)$, and $\left(V(n)[C]\right)^{p-1}\equiv 1\pmod p$ by Fermat's Little Theorem. Therefore, we can conclude that $\sigma^{(s'+s)(p-2)+s}(V(n))$ contains $V(0)$, as desired.
\end{proof}

Note that $t_0$ above does not depend on $n$. We now have everything we need to prove our main result and more.

\begin{thm}\label{main}
If $P$ and $Q$ are Laurent polynomials with integer coefficients and $p$ is prime, then the sequence $\ct{P^nQ}\bmod p$ is linearly recurrent if and only if $p\nmid \ct{P^n}$ for all $n\in\N$, and this can be checked for bounded $n<B_{P,p}$. When the sequence is not linearly recurrent, it is $0$ with frequency $1$. This classification is independent of $Q$.
\end{thm}
\begin{proof}
If there exists an $n\in\N$ such that $p\mid \ct{P^n}$, then Proposition \ref{check_bound} tells us that there is such an $n<B_{P,p}$. And Theorem \ref{zero_density} tells us that $\ct{P^nQ}\bmod p$ is not uniformly recurrent.\\

For the converse, we show that $\alpha$ from Definition \ref{morph} is linearly recurrent, from which the result follows since our sequence is the image of $\alpha$ under $\code_Q$. Theorem 10.8.4 of \cite{ShallitBook} states that if a $p$-automatic sequence is uniformly recurrent, it is linearly recurrent. Thus, it is sufficient for us to show that $\alpha$ is uniformly recurrent.\\

If $w$ appears as a word in $\alpha$, and its first occurrence ends by index $i<p^{t_1}$, then $w$ appears in $\sigma^{t_1}(V(0)) = V(0)V(1)\cdots V(p^{t_1}-1)$ since $\sigma$ is $p$-uniform. We can conclude from Proposition \ref{v0_appears} that for all $n\in\N$, and all words $w$ appearing in $\alpha$, $$\sigma^{t_1 + t_0}(V(n)) = \sigma^{t_1}\left(x_0V(0)y_0\right) = \sigma^{t_1}(x_0)\sigma^{t_1}(V(0))\sigma^{t_1}(y_0) = \sigma^{t_1}(x_0)x_1wy_1\sigma^{t_1}(y_0).$$
In other words, $w$ appears in every $\sigma^{t_1 + t_0}(V(n))$, and since $$\alpha = \sigma^{t_1 + t_0}(\alpha) = \sigma^{t_1 + t_0}(V(0))\sigma^{t_1 + t_0}(V(1))\sigma^{t_1 + t_0}(V(2))\cdots$$ and $\sigma$ is uniform (so that $\sigma^{t_1 + t_0}$ is also uniform), this implies that $\alpha = V(0)V(1)V(2)\cdots$ is uniformly recurrent.
\end{proof}

\begin{cor}\label{combinations}
Let $F: \F_p^\ell\rightarrow X$ (which is deterministic) for some $\ell\in\N$ and let $b_n = F\left(\ct{P^{n+k_1}Q_1}\bmod p, \ldots, \ct{P^{n+k_\ell}Q_\ell}\bmod p\right)$ where the $k_i$ are arbitrary non-negative integers and the $Q_i$ are arbitrary Laurent polynomials. Then $b_n$ is linearly recurrent if and only if for all $n\in\N$, $p\nmid \ct{P^n}$. Furthermore, in the case where there exists $n$ such that $p\mid \ct{P^n}$, then the frequency of $F(0,\ldots,0)$ in $b_n$ is $1$. For example, if $F$ takes a linear combination of its inputs then we get that $b_n = \sum_{i=0}^\ell \beta_i\cdot\left(\ct{P^{n+k_i}Q_i}\bmod p\right)$ for arbitrary $\beta_i$ is linearly recurrent if and only if $p\nmid \ct{P^n}$ for all $n$.
\end{cor}
\begin{proof}
The proof of Theorem \ref{main} actually shows that $\alpha = V(0)V(1)V(2)\cdots$ is linearly recurrent. Since each word beginning with $b_{n_0}$ of length $L$ is determined by the word $V(n_0-\min(k_i))\cdots V(n_0+\max(k_i)+L)$, we can conclude that the sequence $b_n$ is linearly recurrent when $\alpha$ is. Furthermore, if $\alpha$ is not linearly recurrent, it contains $0$ vectors with frequency $1$ (by Theorem \ref{zero_density}), and hence $b_n$ contains $F(0,\ldots,0)$ with frequency $1$.
\end{proof}

We conclude this section by noting that even in the case where $p\mid \ct{P^{n_0}}$ for some $n_0$, resulting in $\alpha$ containing $0$ vectors with frequency $1$, we still have that $\alpha$ is recurrent (see below) so long as there exists an $n$ such that $p\nmid \ct{P^n}$.

\begin{definition}
A sequence $s_n$ is called \emph{recurrent} if for every occurrence of a word, $w = s_is_{i+1}\cdots s_{i+\ell-1}$, there is another later occurrence. I.e., there is a $j > 0$ such that $w = s_{i+j}s_{i+j+1}\cdots s_{i + j + \ell - 1}$. This is equivalent to saying that every word that appears recurs infinitely often, but there is no bound on gaps between occurrences.
\end{definition}

\begin{prop}\label{recurrent}
If $P$ and $Q$ are Laurent polynomials and $p$ is prime, and there exists some $n_0>0$ such that $p\nmid \ct{P^{n_0}}$, then the sequence $b_n = \ct{P^nQ}\bmod p$ is recurrent.
\end{prop}
\begin{proof}
Proposition \ref{v0_appears} assures us that $\sigma^{t_0}(V(n_0))$ contains $V(0)$. Therefore, if $w$ is a word in $\alpha$ and we consider an occurrence of $w$ that ends by index $i < p^{t_2}$, then $w$ appears in $\sigma^{t_2}(V(0))$ and thus $w$ appears in $\sigma^{t_2 + t_0}(V(n_0))$. Finally, since $$\alpha = \sigma^{t_2 + t_0}(\alpha) = \sigma^{t_2 + t_0}(V(0))\sigma^{t_2 + t_0}(V(1))\cdots\sigma^{t_2 + t_0}(V(n_0))\cdots$$ we can conclude that $w$ recurs and consequently the corresponding words in $b_n$ do as well. Specifically, the occurrence of $w$ we began with is in the first segment, $\sigma^{t_2+t_0}(V(0))$, of the above decomposition of $\alpha$ and there is a subsequent occurrence in the $\sigma^{t_2+t_0}(V(n_0))$ segment. Note that the obstruction to uniform recurrence is the dependence of $t_2$ on the index of the occurrence of $w$ since there can be arbitrarily large gaps between constructed occurrences.
\end{proof}

\begin{question}
What can be said about the recurrence (or non-recurrence) of $\ct{P^n}$ when $\ct{P^n}\equiv 0$ for all $n>0$? For example, this is true for $P(x) = x^{p-1} + x^{-1}$.
\end{question}

\section{Extending Results to Prime Powers}
Analogous results to Theorem \ref{main} and Corollary \ref{combinations} hold when we consider constant term sequences modulo general prime powers, but some extra care is required in this case. To avoid unnecessary confusion in the case of primes (as presented above), the discussion of prime powers has been excised to this section.\\

The construction of the Rowland-Zeilberger automaton in \cite{rowlandzeilberger} becomes mildly more complicated mod $p^a$ for $a>1$ because Lemma \ref{RZ_cong} no longer holds. However, after reading fewer than $a$ base-$p$ digits, we have an analogous result.

\begin{definition}
When the Laurent polynomial $P$ with coefficients in $\Z/p^a\Z$ is fixed, we define $\tilde{P}(\x)$ to be such that $P(x_1, \ldots, x_r)^{p^{a-1}} \equiv \tilde{P}(x_1^{p^\ell}, \ldots, x_r^{p^\ell})\pmod{p^a}$ where $\ell$ is taken to be as large as possible. Note that since $P(x_1, \ldots, x_r)^p\equiv P(x_1^p,\ldots, x_r^p)\pmod p$, it follows that $$P(x_1,\ldots, x_r)^{p^a} = (P(x_1,\ldots, x_r)^p)^{p^{a-1}}\equiv P(x_1^p, \ldots, x_r^p)^{p^{a-1}}\pmod {p^a}.$$ Thus, $\tilde{P}(x_1^{p^\ell}, \ldots, x_r^{p^\ell})^p\equiv \tilde{P}(x_1^{p^\ell\cdot p}, \ldots, x_r^{p^\ell\cdot p})\pmod{p^a}$, or equivalently, by a change of variables, $$\tilde{P}(x_1,\ldots, x_r)^p\equiv \tilde{P}(x_1^p,\ldots, x_r^p)\pmod{p^a}.$$
Consequently, $\tilde{P}$ is stable under Rowland-Zeilberger transitions, i.e., $\tilde{\tilde{P}} = \tilde{P}$. Also note that when $a=1$, this section coincides with the previous one if we take $\tilde{P} = P$.
\end{definition}

\begin{lemma}\label{zero_conversion}
For all Laurent polynomials $P$ and primes $p$, there exists an $n_0\in\N$ such that $p\mid\ct{P^{n_0}}$ if and only if there exists an $n_1\in\N$ such that $p\mid\ct{\tilde{P}^{n_1}}$. Furthermore, if $n_0$ exists one can take $n_1$ to be $n_0$.
\end{lemma}
\begin{proof}
Because $\ct{\tilde{P}(x_1,\ldots, x_r)^n} = \ct{\tilde{P}(x_1^{p^\ell}, \ldots, x_r^{p^\ell})^n} = \ct{P(x_1,\ldots, x_r)^{p^{a-1}n}}$, the sequence $(\ct{\tilde{P}^n})_{n\in\N}$ is a subsequence of $(\ct{P^n})_{n\in\N}$. Thus, it is sufficient to show that when $\ct{P^{n_0}}\not\equiv 0\pmod p$, then $\ct{\tilde{P}^{n_0}}\not\equiv 0\pmod p$.\\

By Lemma \ref{zero_out} and Definition \ref{Vdef} we have that $\gamma(0)^s\cdot V(n_0)\neq 0$. Furthermore, using Lemmas \ref{zero_out} and \ref{Vlemma}, $$\gamma(0)^s\cdot V(n_0) = \gamma(0)^{s+a-1}\cdot V(n_0) = \gamma(0)^s\cdot \gamma(0)^{a-1}\cdot V(n_0) = \gamma(0)^s\cdot V(p^{a-1}n_0)\neq 0,$$ from which we can conclude that $\ct{\tilde{P}^{n_0}}\equiv \ct{P^{p^{a-1}n_0}}\not\equiv 0\pmod p$.
\end{proof}

We now apply some simple tricks to describe $\ct{P^nQ}\pmod{p^a}$ in terms of sequences of the form $\ct{\tilde{P}^nQ}\pmod{p^a}$. And for $\tilde{P}$, since $\ct{\tilde{P}(x_1,\ldots, x_r)^p}\equiv \ct{\tilde{P}(x_1^p, \ldots, x_r^p)}\pmod{p^a}$, all of our completed analysis applies so long as we substitute $\F_p$ coefficients with coefficients from $\Z/p^a\Z$ and various symbols with their (mod $p^a$)-counterparts, which are labeled with the subscript $2$ in this section. These translations are so direct that we omit proofs for most of the results and instead reference the analogous results above.\\

\begin{prop}\label{conversion}
For all Laurent polynomials $P$ and $Q$ with coefficients in $\Z/p^a\Z$, all $n\in\N$, and all $k\in\Z/p^{a-1}\Z$, $$\ct{P^{p^{a-1}n+k}Q}\equiv \ct{\tilde{P}^n\cdot\Lambda_{p^\ell}(P^kQ)}\pmod{p^a}.$$
\end{prop}
\begin{proof}
\begin{align*}
\ct{P(x_1, \ldots, x_r)^{p^{a-1}n+k}Q(x_1, \ldots, x_r)} &= \ct{(P(x_1, \ldots, x_r)^{p^{a-1}})^nP(x_1, \ldots, x_r)^kQ(x_1, \ldots, x_r)}\\
&\equiv \ct{\tilde{P}(x_1^{p^\ell}, \ldots, x_r^{p^\ell})^nP(x_1, \ldots, x_r)^kQ(x_1, \ldots, x_r)}\\
&\equiv \ct{\tilde{P}(x_1, \ldots, x_r)^n\Lambda_{p^\ell}(P(x_1, \ldots, x_r)^kQ(x_1, \ldots, x_r))}.
\end{align*}
\end{proof}

Now we can compute $\ct{P^{n'}Q}\bmod{p^a}$ by reading the $a-1$ least significant base-$p$ digits of $n'$, call this number $k$, and then use the digits remaining, call this number $n$, and then compute $\ct{\tilde{P}^n\cdot\Lambda_{p^\ell}(P^kQ)}\bmod{p^a}$ where $\tilde{P}$ satisfies $\tilde{P}(x_1, \ldots, x_r)^p\equiv \tilde{P}(x_1^p, \ldots, x_r^p)\pmod{p^a}$. Using this congruence, we obtain lemmas analogous to Lemma \ref{RZ_cong} and Lemma \ref{deg_bound} for $\tilde{P}$:

\begin{lemma} For all $k\in\F_p$, $n\in\N$ and all Laurent polynomials, $\tilde{P}(\x)$, with coefficients in $\Z/p^a\Z$ satisfying $\tilde{P}(x_1, \ldots, x_r)^p\equiv \tilde{P}(x_1^p, \ldots, x_r^p)\pmod{p^a}$, $$\ct{\tilde{P}^{pn + k}Q}\equiv \ct{\tilde{P}^n\cdot\Lambda_p(\tilde{P}^kQ)}\pmod{p^a}.$$
\end{lemma}

\begin{lemma}
For a fixed $\tilde{P}$, $$\tilde{m} := \max(\deg (\tilde{P}) - 1, \deg (Q_0))\geq \deg\Lambda_p^{k_t}(\cdots \Lambda_p^{k_0}(Q_0)\cdots ).$$ That is, $\tilde{m}$ is a bound on the degree of every polynomial that is the result of iterating $\Lambda_p^k$ on input $Q_0$ for all values of $k\in\F_p$. In particular, $\tilde{m}$ bounds the degree of every $Q_0$ appearing in the Rowland-Zeilberger construction.
\end{lemma}

Note that by Proposition \ref{conversion}, the maximum that defines $\tilde{m}$ must be taken over $\deg(\tilde{P})-1$ and $Q_0 = \deg(\Lambda_{p^\ell}(P^kQ))$ for all $k\in\Z/p^{a-1}\Z$ for all values of $Q$ we are considering.

\begin{definition}
We wish to encode Laurent polynomials with coefficients in $\Z/p^a\Z$ of degree $\leq m$ as a product of copies of $\Z/p^a\Z$, and so we must choose a computational basis: Impose an order on $\tilde{T} = [-\tilde{m},\tilde{m}]^r$ (e.g. lexicographical), where $\tilde{T}$ indexes the set of coefficients in a polynomial with degree $\leq \tilde{m}$. If $Q = \sum_{i\in \tilde{T}}q_i\x^i$ then let $\text{vec}(Q)$ be the row vector $(q_i)_{i\in \tilde{T}}$ and let $\code_Q:((\Z/p^a\Z)^{\tilde{T}})^*\rightarrow(\Z/p^a\Z)^*$ be the coding defined by $\code_Q(\vec{u}) = \text{vec}(Q)\cdot\vec{u}$.
\end{definition}

Just as in Definition \ref{linrepdef}, the map $Q\mapsto \Lambda_p(\tilde{P}^kQ)$ is $(\Z/p^a\Z)$-linear and has the following matrix description:

\begin{definition}
If $\tilde{\gamma}(k) := (\coef{pj-i}{\tilde{P}^k})_{i,j\in \tilde{T}}$, then $\text{vec}(Q)\cdot\tilde{\gamma}(k) = \text{vec}(\Lambda_p(\tilde{P}^kQ))$. Defining $\tilde{\gamma}$ as a matrix-valued string morphism in this way yields the $p$-linear representation $(\text{vec}(Q), \tilde{\gamma}, \tilde{V}(0))$ for the sequence $\left(\ct{\tilde{P}^nQ}\bmod {p^a}\right)_{n\in\N}$, where the only non-zero entry of $\tilde{V}(0)$ is at the index corresponding to the constant term, where it is $1$. More generally, define $\tilde{V}(n)$ to be the column vector $(\ct{\tilde{P}^n\x^i}\bmod{p^a})_{i\in \tilde{T}}\in (\Z/p^a\Z)^{\tilde{T}}$.
\end{definition}

Just as with Lemma \ref{Vlemma}:

\begin{lemma}
For all non-negative integers $n$ and all $k\in\F_p$, $$\tilde{\gamma}(k)\cdot \tilde{V}(n) = \tilde{V}(pn + k).$$
\end{lemma}

\begin{definition}
Just as in Definition \ref{morph}, we use $\tilde{\gamma}$ to define the $p$-uniform string morphism $\tilde{\sigma}: \left((\Z/p^a\Z)^{\tilde{T}}\right)^*\rightarrow \left((\Z/p^a\Z)^{\tilde{T}}\right)^*$ by $$\tilde{\sigma}(\vec{u}) := (\tilde{\gamma}(0)\cdot\vec{u})(\tilde{\gamma}(1)\cdot\vec{u})\cdots (\tilde{\gamma}(p-1)\cdot\vec{u}).$$ Then define $\tilde{\alpha} := \tilde{\sigma}^\omega(\tilde{V}(0)) = \tilde{V}(0)\tilde{V}(1)\tilde{V}(2)\cdots\in \left((\Z/p^a\Z)^{\tilde{T}}\right)^*$.
\end{definition}

Just as with Lemma \ref{zero_out}:

\begin{lemma}
Let $C$ be the index of the polynomial $1$ in $\tilde{T}$'s order. There is an integer $s\in\N$ such that for every $s'\geq s\in\N$, $\tilde{\gamma}(0)^s = \tilde{\gamma}(0^s) = e_{C,C}$, that is to say, it has a $1$ in the row and column corresponding to the constant term and $0$s everywhere else.
\end{lemma}

In analogy with Theorem \ref{zero_density}:

\begin{lemma}\label{zero_density_power}
If $\tilde{P}$ is a Laurent polynomial satisfying $\tilde{P}(x_1,x_2,\ldots,x_r)^p\equiv \tilde{P}(x_1^p,x_2^p,\ldots,x_r^p)\pmod{p^a}$ and there exists some $n_0\in\N$ such that $\ct{\tilde{P}^{n_0}}\equiv 0\pmod p$, then the frequency of $0$ is $1$ in every sequence of the form $\ct{\tilde{P}^nQ}\bmod {p^a}$ for every Laurent polynomial $Q$.
\end{lemma}
\begin{proof}
If $p\mid \ct{\tilde{P}^{n_0}}$, then $\tilde{\gamma}(0)^s\tilde{\gamma}(n_0)\tilde{\gamma}(0)^s$ is a zero-divisor multiple of $e_{C,C}$, so that there exists some $t$ such that $\left(\tilde{\gamma}(0)^s\tilde{\gamma}(n_0)\tilde{\gamma}(0)^s\right)^{t} = 0$. Since $\ct{\tilde{P}^nQ} \equiv \text{vec}(Q)\cdot \tilde{V}(n) = \text{vec}(Q)\cdot\tilde{\gamma}(n)\cdot \tilde{V}(0)$ all $n$ such that $n_p$ contains $(0^s(n_0)_p0^s)^{t}$ must have $\ct{\tilde{P}^nQ} \equiv 0\pmod{p^a}$.\\

If $\beta = \vert (n_0)_p\vert$, then consider each $n\in \left[0, p^{t(\beta+2s)\cdot k}\right)$ as corresponding to an element of $\left(\F_p^{t(\beta+2s)}\right)^k$ via its base-$p$ representation chunked into blocks of size $t(\beta+2s)$. If any of the $k$ entries for $n$ correspond to $(0^s(n_0)_p0^s)^{t}$, then we know that $\tilde{V}(n)=0$ and $\ct{\tilde{P}^nQ}\equiv 0$, so at least $p^{t(\beta+2s)\cdot k} - \left(p^{t(\beta+2s)} - 1\right)^k$ of such $n$ have $\ct{\tilde{P}^nQ} \equiv 0$. Therefore, the frequency of $0$ in $\ct{\tilde{P}^nQ}\bmod {p^a}$ is at least $\frac{p^{t(\beta+2s)\cdot k} - \left(p^{t(\beta+2s)} - 1\right)^k}{p^{t(\beta+2s)\cdot k}} = 1 - \left(\frac{p^{t(\beta+2s)} - 1}{p^{t(\beta+2s)}}\right)^k\rightarrow 1$ as $k\rightarrow\infty$. In fact, this shows that the frequency of $0$ in $\tilde{\alpha}$ is $1$.
\end{proof}

The last thing we need before we can prove our main results is the following proposition analogous to Proposition \ref{v0_appears}:

\begin{prop}\label{v0_appears_power}
There exists $t_0\in\N$ such that for all $n\in\N$, if $p\nmid \ct{\tilde{P}^n}$ then $\tilde{\sigma}^{t_0}(\tilde{V}(n)) = x_0\tilde{V}(0)y_0$ with $x_0,y_0\in\left((\Z/p^a\Z)^{\tilde{T}}\right)^*$. That is, $\tilde{V}(0)$ is a character in $\tilde{\sigma}^{t_0}(\tilde{V}(n))$ for every such $n$.
\end{prop}
\begin{proof}
Let $s\in\N$ such that $\tilde{\gamma}(0)^s = e_{C,C}$. Next, fix $s'\in\N$ such that $C$-term entries (i.e., indexed by $C$) of $\tilde{\sigma}^{s'}(\tilde{V}(0))$ contain all $C$-term values of $\tilde{V}(i)$ that ever occur for any $i\in\N$. Recall that the $C$-term value of $\tilde{V}(i)$ is the constant coefficient of $\tilde{P}^i$. Note that $\tilde{V}(n)[C]\not\equiv 0\pmod{p^a}$ by our assumption that $p\nmid \ct{\tilde{P}^n}$. Finally, we have that $\tilde{\gamma}(0)^s\cdot \tilde{V}(n) = \left(\tilde{V}(n)[C]\right)\tilde{V}(0)$, $\tilde{\sigma}^{s' + s}(\tilde{V}(0))$ contains $\left(\tilde{V}(n)[C]\right)\tilde{V}(0)$, and $\left(\tilde{V}(n)[C]\right)^{\varphi(p^a)} = 1$ by Euler's theorem (where $\varphi$ is Euler's totient function). Therefore, we can conclude that $\tilde{\sigma}^{(s'+s)(\varphi(p^a)-1)+s}(\tilde{V}(n))$ contains $\tilde{V}(0)$, as desired.
\end{proof}

Note that $t_0$ does not depend on $n$. Finally, our main result:

\begin{thm}\label{main_power}
If $P$ and $Q$ are Laurent polynomials with integer coefficients and $p^a$ is a prime power, then the sequence $\ct{P^nQ}\bmod {p^a}$ is linearly recurrent if and only if $p\nmid \ct{P^n}$ for all $n\in\N$, and this can be checked for bounded $n<B_{P,p}$. When the sequence is not linearly recurrent, it is $0$ with frequency $1$. This classification is independent of $Q$, as well as the exponent $a$.
\end{thm}
\begin{proof}
For every $k\in\Z/p^{a-1}\Z$ and Laurent polynomial $Q$, define $$\code_{Q, k}(\vec{u}) = \text{vec}(\Lambda_{p^\ell}(P^kQ))\cdot \vec{u}$$ where $\vec{u}$ is from $(\Z/p^a\Z)^{\tilde{T}}$. Then for Laurent polynomials $P$ and $Q$, define the coding $$\code_{P, Q}(\vec{u}) = (\code_{Q,0}(\vec{u}), \code_{Q,1}(\vec{u}), \ldots, \code_{Q,p^{a-1}-1}(\vec{u}))\in (\Z/p^a\Z)^{p^{a-1}}.$$ Let $\iota: \left((\Z/p^a\Z)^{p^{a-1}}\right)^*\rightarrow \left(\Z/p^a\Z\right)^*$ denote the morphism determined by the canonical injection on letters, $$\iota(j_1,j_2,\ldots,j_{p^a-1}) = j_1j_2\cdots j_{p^a-1}.$$ Then, by Proposition \ref{conversion}, the sequence $$\beta := \iota(\code_{P, Q}(\tilde{\alpha}))$$ has $\ct{P^nQ}\bmod{p^a}$ as its $n$th term.\\

If there is some $n_0\in\N$ such that $p\mid\ct{P^{n_0}}$ then Lemma \ref{zero_conversion} gives us an $n_1\in\N$ such that $p\mid\ct{\tilde{P}^{n_1}}$ so that Lemma \ref{zero_density_power} tells us that the frequency of $0$ is $1$ in $\tilde{\alpha}$ and thus $0$ also has frequency $1$ in $\beta$ so that, in particular, $\beta$ is not uniformly recurrent.\\

If there is no such $n_0$, and if $w$ appears as a word in $\beta$, then we can extend $w$ to a word $w'$ in $\beta$ corresponding to a word in $\tilde{\alpha}$. That is, $w'$ contains $w$ but starts on an index that is a multiple of $p^{a-1}$ and ends on an index one less than a multiple of $p^{a-1}$. Thus, it is sufficient to show that $\tilde{\alpha}$ is uniformly recurrent as this forces uniform recurrence for $w'$ and consequently for $w$.\\

Given a word $w$ in $\tilde{\alpha}$, choose $t_1$ such that $w$ has its first occurence ending before the index $p^{t_1}$. Thus, $w$ appears in $\tilde{\sigma}^{t_1}(\tilde{V}(0))$. We can conclude from Proposition \ref{v0_appears_power} that for all $n\in \N$, $$\tilde{\sigma}^{t_1 + t_0}(\tilde{V}(n)) = \tilde{\sigma}^{t_1}(x_0\tilde{V}(0)y_0) = \tilde{\sigma}^{t_1}(x_0)x_1wy_1\tilde{\sigma}^{t_1}(y_0).$$ In other words, $w$ appears in every $\tilde{\sigma}^{t_1 + t_0}(\tilde{V}(n))$, and since $$\tilde{\alpha} = \tilde{\sigma}^{t_1+t_0}(\tilde{\alpha}) = \tilde{\sigma}^{t_1 + t_0}(\tilde{V}(0))\tilde{\sigma}^{t_1 + t_0}(\tilde{V}(1))\tilde{\sigma}^{t_1 + t_0}(\tilde{V}(2))\cdots$$ and $\tilde{\sigma}$ is uniform (so that $\tilde{\sigma}^{t_1 + t_0}$ is also uniform), this implies that $\tilde{\alpha}$ is uniformly recurrent, and in turn $\beta$ is uniformly recurrent. Finally, Theorem 10.8.4 of \cite{ShallitBook} tells us that $p$-automatic uniformly recurrent sequences are linearly recurrent.
\end{proof}

Counterintuitively, this theorem shows that $\ct{P^nQ}\bmod {p^a}$ is either linearly recurrent or else $0$ has frequency $1$ for all values $a\geq 1$ simultaneously. Furthermore, just as in Corollary \ref{combinations}, the same generalization applies to prime powers:

\begin{cor}\label{combinations_power}
Let $F: (\Z/p^a\Z)^\ell\rightarrow X$ (which is deterministic) for some $\ell\in\N$ and let $b_n = F\left(\ct{P^{n+k_1}Q_1}\bmod {p^a}, \ldots, \ct{P^{n+k_\ell}Q_\ell}\bmod {p^a}\right)$ where the $k_i$ are arbitrary non-negative integers and the $Q_i$ are arbitrary Laurent polynomials. Then $b_n$ is linearly recurrent if and only if for all $n\in\N$, $p\nmid \ct{P^n}$. Furthermore, in the case where there exists $n$ such that $p\mid \ct{P^n}$, then the frequency of $F(0,\ldots,0)$ in $b_n$ is $1$.
\end{cor}

Lastly, using the same construction of $\beta$ as in the proof of Theorem \ref{main_power}, we have the analogous result to Proposition \ref{recurrent} where we can refer to $P$ and not $\tilde{P}$ because we have Lemma \ref{zero_conversion}:

\begin{prop}
If $P$ and $Q$ are Laurent polynomials and $p^a$ is a prime power, and there exists some $n_0>0$ such that $\ct{P^{n_0}}\bmod p\neq 0$, then the sequence $b_n = \ct{P^nQ}\bmod {p^a}$ is recurrent.
\end{prop}

\section{Application to Representation Theory}
In this section, we briefly interpret Theorem \ref{main_power} (or more precisely, Corollary \ref{combinations_power}) in the context of sequential tensor powers of semisimple Lie algebra representations. This context is a great source of constant term sequences.\\

Let $\mathfrak{g}$ be a semisimple Lie algebra, and let $V$ be a representation of $\mathfrak{g}$ whose formal character is $P\in \Z\left[\x,\x^{-1}\right]$. If $W$ is an irreducible representation of $\mathfrak{g}$, then let $a_{W, n}$ denote the multiplicity of $W$ in the direct sum decomposition of $V^{\otimes n}$, which has character $P^n$. It is a consequence of the Weyl Integral Formula (see, for example, Section IV.1 of \cite{reptheory}) that a multiple of $a_{W, n}$ is a constant term sequence; in particular, if $Q$ is the product of the formal character of $W$ with the square of the Weyl denominator, then $C\cdot a_{W,n} = \ct{P^nQ}$, where $C$ is the order of the Weyl group. The following theorem follows as a result of Corollary \ref{combinations_power}.

\begin{thm} Let $\mathfrak{g}$ be a semisimple Lie algebra whose irreducible representations are $\{W_i\}$ and let $V$ be a representation of $\mathfrak{g}$ whose formal character is $P\in\Z\left[\x,\x^{-1}\right]$. Let $a_{i,n}$ denote the multiplicity of $W_i$ in the direct sum decomposition of $V^{\otimes n}$. Then for each prime $p$, either all of the sequences $(a_{i,n}\bmod {p^a})_{n\in\N}$ are simultaneously linearly recurrent for every $a\geq 1$ and every $i$ or else all of those sequences simultaneously have $0$ with frequency $1$. In particular, the sequences are linearly recurrent if and only if there exists an $n_0<B_{P,p}$ such that $p\mid \ct{P^{n_0}}$.
\end{thm}
\pagebreak

\bibliographystyle{plain}
\bibliography{LinRepForConstantTerms}{}

\end{document}